\documentclass[sts]{imsart}

\RequirePackage{amsthm,amsmath,amsfonts,amssymb}
\RequirePackage[authoryear]{natbib}
\RequirePackage{jr}
\usepackage{comment}
\usepackage{graphicx}

\startlocaldefs

\newtheorem{theorem}{Theorem}[section]
\newtheorem{Theorem}{Theorem}[section]
\newtheorem{lemma}[theorem]{Lemma}

\newtheorem{Proposition}[Theorem]{Proposition}
\newtheorem{proposition}[theorem]{Proposition}

\newtheorem{remark}{Remark}[section]

\newtheorem{example}{Example}[section]

\def\sure{{\rm SURE}\xspace}
\def\th{\vartheta}
\endlocaldefs

\begin{document}

\begin{frontmatter}
\title{No need for an oracle:  the nonparametric maximum likelihood decision in the compound decision problem is minimax}
\runtitle{CD Minimax }

\begin{aug}
\author[A]{\fnms{Ya\hspace{-0.1em}{'}\hspace{-0.1em}acov}~\snm{Ritov}\ead[label=e1]{yritov@umich.edu}}
\thankstext{t1}{Supported in part by NSF Grant DMS-2113364.  This paper follows the author's Blackwell's Lecture, JSM 2023.}
\address[A]{\jr\ is Professor, Department of Statistics, University of Michigan, Ann Arbor, Michigan USA\printead[presep={\ }]{e1}.}
\end{aug}

\begin{abstract}
  We discuss the asymptotics of the nonparametric maximum likelihood estimator (NPMLE) in the normal mixture model. We then prove the convergence rate of the NPMLE decision in the empirical Bayes problem with normal observations. We point to (and heavily use) the connection between the NPMLE decision and Stein unbiased risk estimator (\sure).
  Next, we prove that the same solution is optimal in the compound decision problem where the unobserved parameters are not assumed to be random.

  Similar results are usually claimed using an oracle-based argument. However, we contend that the standard oracle argument is not valid. It was only partially proved that it can be fixed, and the existing proofs of these partial results are tedious. Our approach, on the other hand,  is straightforward and short.
\end{abstract}

\begin{keyword}
\kwd{Empirical Bayes}
\kwd{Compound decision}
\kwd{Minimax}
\kwd{Oracle}
\kwd{Nonparametric maximum likelihood}
\end{keyword}

\end{frontmatter}

\section{Introduction}

Suppose, for some \emph{unobserved} $\th_1,\dots,\th_n$, the observations $Y_1,\dots,Y_n$ are independent given $\th_1,\dots,\th_n$ and   $Y_i\mid \th_{[n]} \dist N(\th_i,\sig^2)$.  The goal is to estimate $\th_{[n]}$ under the  $L_2$ norm,
 \eqsplit[L2loss]{
    \scl(\th_{[n]},\hat\th_{[n]})=\frac1n\summ i1n (\hat\th_i-\th_i)^2,
  }
where $[n]\eqdef 1,\dots,n$ and for a general  sequence, $a_{[n]}$ is a short notation for $a_1,\dots,a_n$.  In this note we restrict our attention to the dense case when $\lim n^{-1}\sum \th_i^2$ is bounded away (in probability) from  0.

A few models were considered in the literature to model $\th_{[n]}$.
In the classical empirical Bayes (EB) problem, cf. \cite{Robbins56, Zhang03, Efron19}, $\th_1,\dots,\th_n$ are \iid $G$ for some unknown $G$. More generally, they may be a time sequence following some state-space model or a hidden Markov model. The statistician's aim is to minimize $R(G,\hat \th_{[n]})\eqdef \E \scl(\th_{[n]},\hat\th_{[n]})$ without the knowledge of $G$. Note that the expectation is taken assuming that $\th_{1},\dots,\th_n$ are  random. In the compound decision (CD) problems, $\th_{1},\dots,\th_n$ are just fixed unknown parameters. The target now is to minimize the same risk function, but the expectation is taken only over $\hat\th_{[n]}$ since $\th_{[n]}$ are fixed. In this paper, we restrict our attention only to the EB/CD problem, i.e., $\th_1,\dots.\th_n$ are either \iid sample or unknown but fixed. To simplify the discussion, we assume that $|\th_i|<c$.

\cite{Efron19} titled his authoritative  \emph{Statistical Science} review by \emph{``Bayes, Oracle Bayes and Empirical Bayes.''} Oracle Bayes is what we refer to as the CD problem. Efron's title brings forward an oracle that essentially makes the CD setup as if it were an  EB with respect to the (unknown to the statistician) empirical distribution of the $\th$s. Our main contribution is to argue that this oracle is misleading and unnecessary. We claim that the existence of an asymptotic minimax solution can be proved by the standard method of considering the statistical problem as a zero sum game (with asymmetric information) between the statistician and Nature.

There are two complementary methods to prove that a result is minimax. In both methods, we compare the problem at hand to another problem whose solution is known and serves as a bound of what is achievable. With the oracle based argument, we compare the statistician to a more knowledgeable statistician, who faces an easier problem. In the game theoretic approach, we make Nature more atrocious; thus, the statistician faces a more difficult problem, and we can argue that, in general, he cannot do better.

We advocate for the nonparametric maximum likelihood estimator (NPMLE) as a good strategy for dealing with both models. It is based on the assumption that the observations come from a mixture of normal distributions \emph{as if}  $\th_{[n]}$  are \iid (as, indeed, they are in the EB model but not in the CD one). This approach is based on the connection between the NPMLE and Stein's unbiased risk estimator (\sure), as presented below. We then bring a new proof that the same solution is indeed valid in the CD context, where seemingly the NPMLE doesn't fit the model.

Thus the contribution of this short communication is fourfold: (i) We prove concentration inequality for the NPMLE procedure in the EB and CD problems, (ii) We relate the \sure to the NPMLE, (iii) we argue that the oracle-type argument suggested in the literature is not valid, and (iv) we prove that the NPMLE is minimax in both problems.

\section{Maximum likelihood estimator for the empirical Bayes problem}

Consider the EB problem described in the introduction. If $G=G_0$ was known, the optimal decision was the Bayes estimator $\hat\th_i=\del_{G_0}(Y_i)$, where $\del_G(y)\eqdef\E(\th\mid Y=y)$ under the assumption that $\th\dist G$ and $Y\mid \th\dist \Phi\bigl((\cdot-\th)/\sig\bigr)$, and $\Phi$ is the standard normal cdf. It is well known that $\del_G$  is given by the Tweedie formula, \cite{Robbins56,Efron11,Efron19}:
 \eqsplit[Tweedie]{
    \del_G(y) &= y + \sig^2\frac{f'_G(y)}{f_G(y)},
 }
{where}
\eqsplit{
    f_G(y)&=\frac1\sig\int \varphi\bigl(\frac{y-\th}{\sig}\bigr)dG(\th)
  }
is the marginal density of $Y$. Here $\varphi $ is the standard normal density. The importance of the Tweedie formula is that it expresses the Bayes decision as a function of the observable, $Y_{[n]}$, and not of the unknown distribution $G$.  

Estimating $G$ has a long history, cf.,  \cite{ riceRosenblatt83, ritov1983, carrollHall88,  stefanskiCaroll90, fan1991, donohoLaw92}. However, we are not interested in estimating $G$ per se. We may have interest in estimating the optimal procedure and the risk function, and this can be estimated directly.  In the EB context, where $G$ is not known, \cite{BrownGreenshtein09} proved that we can replace $f_G$ in the Tweedie formula by a kernel density estimate of the marginal distribution of $Y$; \cite{JiangZhang09}, \cite{SahaGunt20}, and
\cite{GreenshteinRitov22} advocated estimating $f_G$ by $f_{\hat G}$, where $\hat G$ is the nonparametric maximum likelihood (NPMLE) of $G$. 

We use the term NPMLE since the object maximizing the likelihood is a general distribution function and thus ``nonparametric.''  However, the class of possible distributions of $Y$ is dominated by the Lebesgue measure and bounded by from above by 
$
   \frac1\sig\varphi\bigl(\frac{y-c}{\sig}\bigr)\vee\frac1\sig\varphi\bigl(\frac{y-c}{\sig}\bigr)\vee\frac{1}c\ind\bigl(|y|<c\bigl)$ from above and 
   $\frac1\sig\varphi\bigl(\frac{y-c}{\sig}\bigr)\wedge\frac1\sig\varphi\bigl(\frac{y-c}{\sig}\bigr)$ from below.
In this sense, it is not a generalized MLE, but a simple maximization of the likelihood (see \eqref{mledisc} below).  Thus we do not need a specific device to define it like the definition suggested in the seminal paper of \cite{kieferWolfowitz1956}. Computationally, this is a mixture model with respect to one parameter exponential family  to which the EM algorithm can be applied. Moreover, we can apply the EM algorithm to maximization over a fixed grid. Personally, this is the algorithm I usually use, and typically, early stopping yields better results. A more efficient algorithm is  given by \cite{KoenkerMizera14}. The maximization is feasible since $\hat G$ has at most $n$ support points and we should maximize
 \eqsplit[mledisc]{
    \ell_n(\ti\th_{[n]},\ti g_{[n]}; Y_{[n]}) &= \summ i1n \log\summ j1n \ti g_j e^{-(Y_i-\ti\th_j)^2/2\sig^2}
  }
 over $\ti\th_{[n]}$ and $\ti g_{[n]}$ in $[-c,c]^n$ and the standard simplex, respectively. See Appendix \ref{app:MLE} for some of the details.

 In the following, in particular Theorem \ref{EBSURE}, I'll make precise the case for our use of the MLE.

We start with:

\begin{Proposition}
\label{prop:boundedDerivatives}
Let $\ell(G;y)=\log f_G(y)+y^2/2\sig^2$. Then $\del_G(y)-y=\sig^2\ell'(G;y)$. For any $G$ with bounded support, the functions $\ell(G;y)$ and $\del_G(y)-y$ have bounded derivatives (with respect to $y$) of any order.
\end{Proposition}

\begin{proof}
Clearly, for an arbitrary $y_0$:
   \eqsplit{
    \ell(G;y) &=   \log \int e^{(y-y_0)\th/\sig^2} e^{y_0\th/\sig^2-\th^2/2\sig^2} dG(\th).
    }
I.e., it is the log of the moment generating function of $d\ti G_{y_0}$ evaluated at $y-y_0$, where $$d\ti G_{y_0}\eqdef  A \exp({y_0\th/\sig^2-\th^2/2\sig^2}) dG(\th)$$ for some normalizing constant $A=A(G,y_0)$. It follows that the $k$th derivative of $\ell(G;y)$ with respect to $y$ at $y_0$ is the $k$ cumulants of $d\ti G_{y_0}$---a distribution with compact support. The $k$th  cumulant of any distributions on $(-c,c)$ is bounded by $(2ck)^k$ (see \cite{DubkovMalakhov76}).

Since
 \eqsplit{
    \frac{\partial}{\partial y}\ell(G;y) &= \frac{f_G'(y)}{f_G(y)}=\del_G(y)-y,
  }
the assertions about $\del_G$ follow.
\end{proof}

In particular:
 \eqsplit{
   |\del_G(y)| &= |\E_G(\th\mid y)| \le |c|;
    \\
   \del_G'(y) &= \var_G(\th\mid y) \in [0, c^2].
  }
Note that $\del_G$ is strictly monotone (unless $G$ is a single point mass and $\del_G$ is constant), which follows since the model has the monotone likelihood property.

If $Y\mid \th$ is a mixture of an exponential family with mixing distribution whose support is everywhere dense, then any mean 0 function of $Y$ is in the tangent space, \cite{bkrw}, Theorem 4.5.1 page 130. It follows that the mean of every smooth function $h(Y)$ is efficiently estimated by its sample mean. It is also efficiently estimated by the NPMLE, and since all semiparametric efficient estimators are equivalent on the $n^{-1/2}$ scale, we obtain:
 \eqsplit[NPMLEmean]{
    \int h(y)d\bbf_n(y) &= \int h(y)dF_{\hat G}(y)+o_p(n^{-1/2}).
  }
The heuristic is that  any function $h(Y)$ in the tangent set can be written as $h(Y)=\E\eta(\th\mid Y)$. If $\eta$ was bounded then  consider the family $dG_t(\th)=\bigl(1+t(\eta(\eta)-E_{\hat G}\eta)\bigr)d\hat G(\th)$, and the log-likelihood as a function of $t$. Since the NPMLE corresponds to $t=0$, the derivative of the log-likelihood at $0$ should be 0, or
 \eqsplit{
    0 &= \frac1n \summ i1n \E_{\hat G}\bigl(\eta(\th)\mid Y_i\bigr) -  E_{\hat G}\eta(\th)
    \\
    &= \int h(y) d\bbf_n(y) - \int h(y) d F_{\hat G}.
  }
The argument can be made precise by considering bounded approximations, see \cite{bkrw}.

The Stein unbiased risk estimator (\sure) follows the following expansion. Suppose $Y\dist N(\th,\sig^2)$, then
 \eqsplit{
    \E\bigl(&\del(Y)-\th)^2
    \\
    &= \E\bigl(\del(Y)-Y\bigr)^2+2\E\Bigl(\bigl(\del(Y)-Y\bigr)(Y-\th)\Bigr)\\
    &\hspace{3em}+\E(Y-\th)^2
    \\
    &=   \E\bigl(\del(Y)-Y\bigr)^2 +2\sig^2 \E\bigl(\del'(Y)-1\bigr)  +\sig^2,
  }
 where the second term on the RHS follows an integration by parts:
  \eqsplit{
    &\hspace{-3em}\int (\del(y)-y)(y-\th)\frac1{\sig}\varphi\bigl(\frac{y-\th}{\sig}\bigr)dy
    \\&= \sig^2 \int (\del'(y)-1)\frac1{\sig}\varphi\bigl(\frac{y-\th}{\sig}\bigr)dy .
   }
Therefore, $n^{-1}\summ i1n \Bigl(\bigl(\del(Y_i)-Y_i\bigr)^2+2\sig^2\bigl(\del'(Y_i)-1\bigr)\Bigr)+\sig^2$ is an unbiased estimator of $R(\th_{[n]},\hat\th_{[n]})$, where $\hat\th_i=\del(Y_i)$. The Stein unbiased risk estimator is given by $\sure(\del,\bbf_n)$, where
 \eqsplit{
    &\hspace{-1em}\sure(\del;F)
    \\
    &= \int (\del(y)-y)^2 dF(y) +  2 \sig^2\int(\del'(y)-1)dF(y)+\sig^2.
  }
In particular $R(G,\del)\equiv \sure(\del,F_G)$, and we are able to express $R(G,\del)$ as a function of the \emph{observed} marginal distribution of $Y$.

    Specializing the definition to the Bayesian estimator $\del_G$, we obtain:
 \eqsplit{
    &\hspace{-2em}\sure(\del_G;\bbf_n)
    \\
    &= \sig^4 \int \Bigl(2\frac{f_G''(y)}{f_G(y)} - \Bigl(\frac{f_G'(y)}{f_G(y)}\Bigr)^2\Bigr) d\bbf_n(y)+\sig^2.
  }

Now, the functions   whose mean are computed in the $\sure(\del_G;\bbf_n)$ are all with uniformly bounded derivatives and with $L_2$ square integrable envelope, since
 \eqsplit{
    0&\le \frac{f_G''(y)}{f_G(y)}\le c^2
    \\
    0&\le \Bigl(\frac{f_G'(y)}{f_G(y)}\Bigr)^2\le 2 y^2+2c^2.
  }
Thus the empirical process  
 \eqsplit{
    \bbe_n(G)=\sqrt n\bigl(\sure(\del_G,\bbf_n)-\sure(\del_G,G_0) \bigr)
  }
converges under $G_0$ to  a Gaussian process, cf., \cite{Stute83} Theorem 1.1: his Conditions (i)--(iii) are satisfied by Proposition \ref{prop:boundedDerivatives}, and since the distribution of $Y$ has sub-Gaussian tails, one can take  $h(t)=\bigl(t(1-t)\bigr)^{1/4}$ as the weight function in his result. In particular,
  \eqsplit[sureEmp]{
    \bbe_n(\hat G)-\bbe_n(G_0)
    &= \OP({J(\|\del_{\hat G}-\del_{G_0}\|_2)}),
   }
where $J(x)=x\log(x)$. Cf. \cite{pollard1984}, ch. VII.

Now, a standard decomposition of the mean square:
 \eqsplit{
    &R ( G_0,\del_{\hat G})
    \\
    &= \E_{G_0} \Bigl(  \E_{G_0}\bigl(\th-\del_{\hat G}(Y)\big)^2\mid Y\bigr)  \Bigr)
    \\
    &=  \E_{G_0} \Bigl( \bigl(\del_{\hat G}(Y)-\del_{G_0}(Y)\bigr)^2 +\var(\th\mid Y)\Bigr)
    \\
    &= R(G_0,\del_{G_0})+  \E_{G_0} \Bigl( \bigl(\del_{\hat G}(Y)-\del_{G_0}(Y)\bigr)^2 \Bigr).
  }
Expressing in terms of \sure:
 \eqsplit[diff2]{
     &\hspace{-3em}\sure(\del_{\hat G},F_{G_0})- \sure(\del_{G_0},F_{G_0}) 
     \\
     &= \E_{G_0} \Bigl( \bigl(\del_{\hat G}(Y)-\del_{G_0}(Y)\bigr)^2 \Bigr)
  }
But by \eqref{sureEmp} and \eqref{NPMLEmean}:  
 \eqsplit[diff1]{
    0&\le \sure(\del_{\hat G},F_{G_0})- \sure(\del_{G_0},F_{G_0})
     \\
     &= \sure(\del_{\hat G},\bbf_n)- \sure(\del_{G_0},\bbf_n)
     \\
     &\hspace{3ex}+ \OP({\frac{J(\|\del_{\hat G}-\del_{G_0}\|_2)}{n^{-1/2}}})
     \\
     &= \sure(\del_{\hat G},\hat G)- \sure(\del_{G_0},\hat G)
     \\
     &\hspace{3ex}+ \OP({\frac{J(\|\del_{\hat G}-\del_{G_0}\|_2)}{n^{-1/2}}})
     \\
     &\le \OP({\frac{J(\|\del_{\hat G}-\del_{G_0}\|_2)}{n^{-1/2}}}),
  }
since $\sure(\del_{\hat G},\hat G)- \sure(\del_{G_0},\hat G)\le 0$. Comparing \eqref{diff1} to \eqref{diff2} we obtain that
  \eqsplit{
    &\sure(\del_{\hat G},F_{G_0})- \sure(\del_{G_0},F_{G_0}) = \OP({{\frac{\log n}{n}}})
   }

In another form:
\begin{theorem}\label{EBSURE}
 \eqsplit{
    R(G_0,\del_{\hat G}) &= R(G_0,\del_{G_0})+ \OP({{\frac{\log n}{n}}}).
  }

\end{theorem}

Theorem \ref{EBSURE} established the asymptotic optimality of using $\del_{\hat G}$. The statistician can achieve without knowing $G$ almost as he could if $G$ was known.

\begin{remark}
The \sure is defined only with respect to the marginal distribution of $Y$, although it implicitly refers to a loss function for some $\th_1,\dots,\th_n$. It is an unbiased estimator of the risk conditioned on $\th_{[n]}$. It obeys uniform LLN and CLT when the set of decision procedures is smooth, as is the case of all Bayes procedures.
\end{remark}

\begin{remark}
Stronger results concerning the NPMLE and its associated risk are given in \cite{JiangZhang09} and \cite{SahaGunt20}. However, our proof is much simpler, points to the connection between the \sure and the Bayesian risk, and serves better our fundamental aim of proving the optimality of the solution to the CD and not to the EB and its proxies.
\end{remark}

\section{The Maximum likelihood estimator and the compound decision problem}

The optimality for the EB problem was proved by considering an \emph{oracle} who can do anything the statistician can and knows everything the statistician knows, but unlike the statistician, also knows\footnote{This isn't the Oracle from Delphi, who was quite limited and obscure human being.}  $G$. Thus, the oracle should use the Bayes procedure $\ti\th_i=\del_G(Y_i)$, and the argument is that the estimator that the statistician is going to use, $\hat\th_i=\del_{\hat G}(Y_i)$ has similar performance as is stated in  Theorem \ref{EBSURE}.

When we consider the CD problem, this oracle is not useful since he assumes that $\th_{[n]}$ are \iid $G$, and they are not. It was suggested in the literature, e.g., \cite{JiangZhang09,Efron19} to consider a similar oracle who knows $\th_{[n]}$ up to permutation, effectively their empirical distribution function $\bbg_n$, and use $\del_{\bbg_n}$. Then it is argued that $\del_{\bbg_n}$ is comparable to $\del_{\hat G}$ (or any other estimator based on the Tweedie formula \eqref{Tweedie}).

 We find this approach to be problematic. An oracle that knows $\bbg_n$ will not use $\del_{\hat G}$. Consider the extreme case of $n=2$, and the oracle knows the values $\th_1<\th_2$, and observe, wlog, $Y_1<Y_2$. In that case, he would consider the pairing $(\th_1,Y_1), (\th_2,Y_2)$ as more likely than the other pairing $(\th_1,Y_2), (\th_2,Y_1)$. Certainly, if $Y_2-Y_2\gg \sig$. \cite{GreenshteinRitov09, GreenshteinRitov19} argue that the minimax decision for this particular oracle is the permutation invariant estimator
 \eqsplit{
    &\hspace{-1em}\del^*_{\bbg_n}(Y_i; Y_{[n]})
    \\
    &= \frac{\summ j1n \th_j \varphi\bigl(\frac{Y_i-\th_j}{\sig}\bigr)\sum_{\pi\in\EuScript{P}_n(i,j)} \prod_{k\ne i} \varphi\bigl(\frac{Y_k-\th_{\pi(k)}}{\sig}\bigr) } {\summ j1n \varphi\bigl(\frac{Y_i-\th_j}{\sig}\bigr)\sum_{\pi\in\EuScript{P}_n(i,j)} \prod_{k\ne i} \varphi\bigl(\frac{Y_k-\th_{\pi(k)}}{\sig}\bigr) },
  }
where $ \EuScript{P}_n(i,j)$ is the set of all permutations of $1,\dots,n$ such that $\pi(i)=j$ for every $\pi\in \EuScript{P}_n(i,j)$. This estimator uses all of $Y_{[n]}$ to estimate $\th_i$.

Thus, the oracle should be somehow prevented from using the optimal (for him) procedure. One approach was to enforce the oracle to use a `simple' or `separable'  estimator such that  $\hat\th_i$ depends on the observations only through $Y_i$. Indeed, an oracle thus restricted should use $\del_{\bbg_n}$. Let us call him the Disabled Oracle. However, we cannot compare him to the statistician: He, the oracle, knows better but is more restricted than the human statistician---the estimator $\del_{\hat G}(Y_i)$ the statistician is using is not simple and uses all the observations (through $\hat G$) for the estimate $\hat\th_i$. The second oracle we considered was the Permutation Oracle. His estimator is efficient, but cannot be used by the human statistician who does not know a permutation of the parameters. The Disabled Oracle, can be copied by the statistician but he does not define a bound, while the Permutation Oracle defined a bound which cannot be achieved. In the next section we will give another argument for the efficiency of what can be achieved which bypasses the need for oracles.

\begin{figure}
  \centering
  \includegraphics[width=0.48\textwidth]{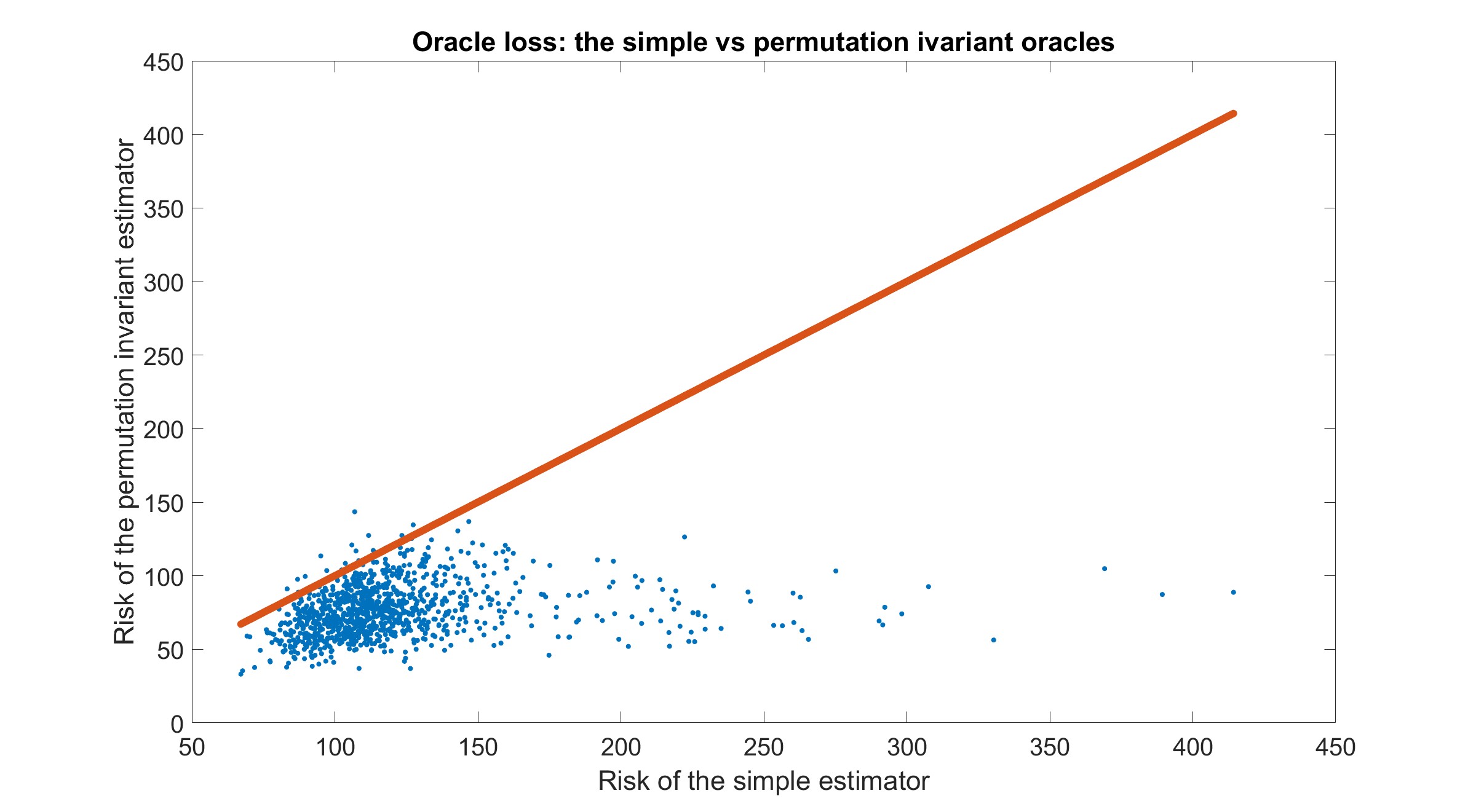}

  (a)
  \includegraphics[width=0.48\textwidth]{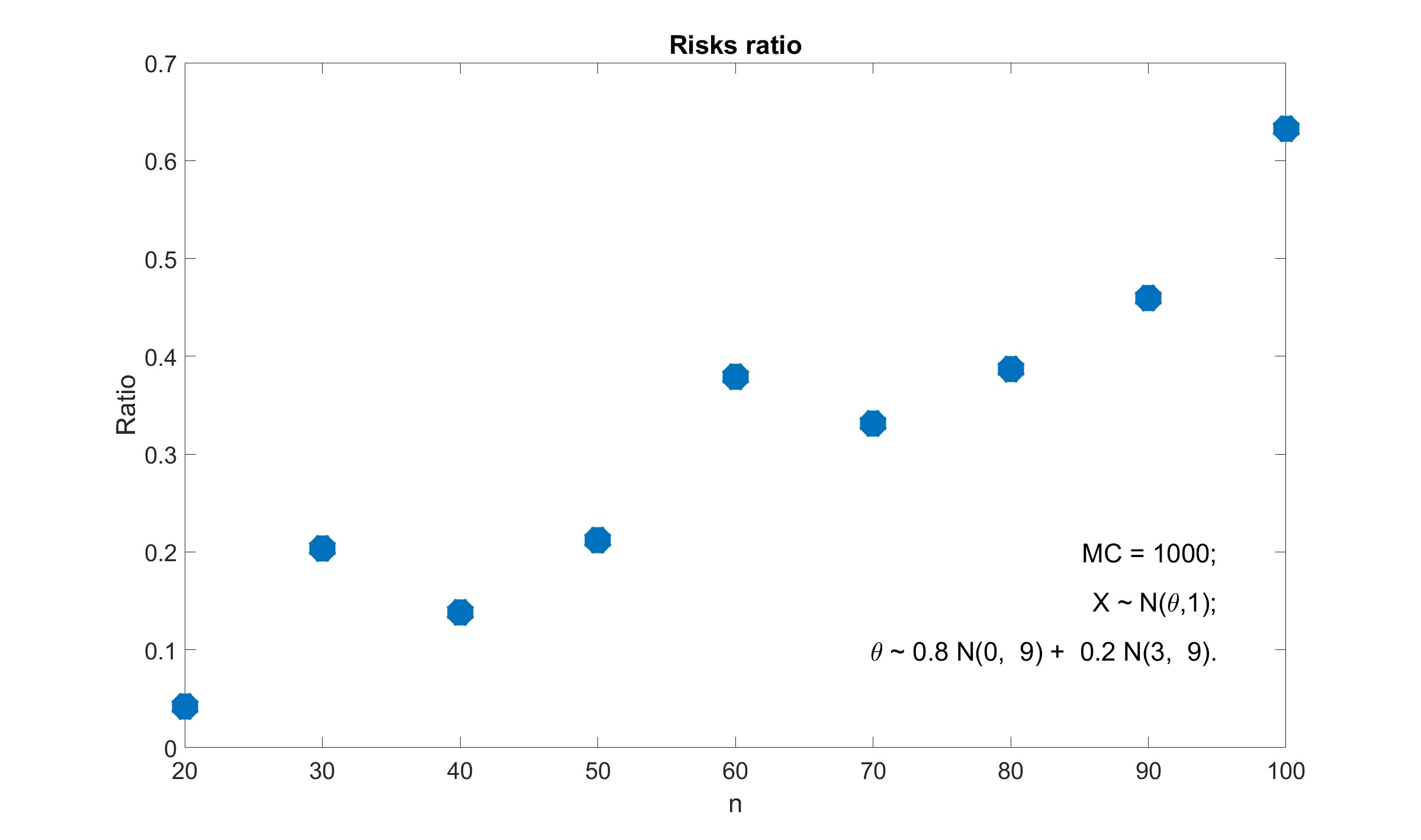}

  (b) 
  \caption{Comparing the risks of the estimators. The permutation invariant estimator is based on 1000000 random permutations. The $Y\dist N(\th,1)$, while $\th$ is with probability 0.8 a $N(0,9)$ rv and with probability 0.2, $N(3,9)$. (a) The loss function of the permutation invariant estimators on 1000 samples vs. the loss of the simple estimators applied to the same samples. The sample size was $n=100$; (b) The efficiency of the permutation invariant estimator relative to that of the simple estimator as a function of sample size $n$. Each point is based on 1000 Monte Carlo trials.}\label{anOracle}
\end{figure}

The difference between $\del_{\bbg_n}$ and $\del^*_{\bbg_n}$ is considerable even with moderate to large sample. The optimal oracle procedure is uncomputable\footnote{The oracle does not really care that the permutation invariant estimator cannot be computed in a reasonable time. However, we are naive human beings and do care about such trivialities. Thus humanoids would not exist long enough to complete the exact computation even for a moderate $n$. }, but we can approximate it by a biased sample of random permutations, where permutations are weighted independently of $Y_{[n]}$, according to their likelihood under $\th_{[n]}$. In Figure \ref{anOracle}, we compare the simple estimator, $\del_{\bbg_n}$, to the biased sample approximation, and the approximate permutation invariant estimator is strictly better up to $n=100$. Similar simulations were presented in \cite{GreenshteinRitov19}. It is true that, asymptotically, the two estimators seem to be equivalent. It was proved in \cite{GreenshteinRitov09} that under specific conditions $R(\bbg_n,\del_{\hat G})=R(\bbg_n,\del^*)+\op({n^{-1}})$.  But the argument is tedious, and the conditions are strong, while it is unclear whether this oracle makes sense.

\begin{example}
  Consider $\th_{i}=\tau_i+\xi_i$ where $\xi_i\dist \scn(0,1)$ and $\tau_i/c_n$ is uniform on $1,\dots,k_n$ for some  $c_n\to\en$ and $n/k_n\to m$. Thus, the oracle is faced by $k_n$ separated clusters, each of approximately size $m$. All the random variables are independent. But, for any $m$, the permutation invariant estimator is strictly better than the simple estimator. We conclude that the argument based on the oracle fails to prove the efficiency of the EB estimator in this CD problem.
\end{example}

To summarize this section, the standard argument of the CD oracle fails, because either the oracle can use a procedure that the statistician cannot mimic or unlike the statistician she is restricted to a simple procedure. We need another argument.

\section{Asymptotic efficiency of the NPMLE for the CD problem}
Our approach is different. There are other approaches to prove the optimality of a procedure. If the oracle method compared the statistician to a better decision maker, the classical approach was to give Nature more freedom and consider a zero-sum game in which the minimizer, the statistician, chooses a procedure and the maximizer,  Nature chooses the parameter maliciously. The optimal strategy for the statistician in this game cannot be improved in general (i.e., over the all parameter space avaiable to Nature).

We consider a game between the Statistician and Nature. If we let Nature choose \emph{any} $\th_{[n]}$, He would select the global minimax solution, cf. \cite{Bickel81}, which would not be relevant to our real statistical realm in which there are some arbitrary $\th_{[n]}$ and not the worst possible. We want to \emph{adapt} to these arbitrary points. To consider adaptive estimator while keeping the minimax notion, we commonly restrict Nature to be in some neighborhood. Thus, we have the local minimax in the sense of H\'ajek  and Le Cam, or more generally, the adaptive estimation in the semiparametric models. We follow these ideas.

\textbf{The game: }For a giving $G$ let $\scg$ be the set of all random (or not) probability distribution functions, such that if $\{\bbg_n\}\in\scg$ then $\E \bbg_n = G$ and   $\bbg_n(t)=\frac1n \summ i1n \ind(\th\le\th_i)$ for some $\th_1,\dots,\th_n$. The statistician observes $Y_{[n]}$ which are independent given $\th_{[n]}$ and $Y_i\dist \scn(\th_i,\sig^2)$. He, the statistician, who neither knows $G$ nor $\bbg_n$, has to choose $\hat\th_{[n]}=\Del(Y_{[n]})$, $\Del:\R^n\to\R^n$, with payoff given by \eqref{L2loss}. The value of the game is $R(\bbg_n,\Del)=E\scl(\th_{[n]},\hat\th_{[n]})$.

Note that $\bbg_n$ in the definition of $R(\cdot,\cdot)$ is a random element. Thus the expectation is also over the (possibly) random $\th_{[n]}$.

\begin{theorem}
  Asymptotically, $\del_{\hat G}$ is a minimax strategy, and $\bbg_n$, an empirical distribution function of a random sample from $G$, is a maxmin strategy for Nature. The conclusion of Theorem \ref{EBSURE} is valid for the game.
\end{theorem}

The proof of the theorem is immediate. If $\bbg_n$ is the edf of an \iid sample from $G$, then $\E\bbg_n=G$, and since for any given procedure of the statistician, the reward for Nature is linear in $\bbg_n$, the value for each $\bbg_n\in\scg$ is the same and is a function only of $G$. On the other hand, if $\bbg_n$ is \iid edf, then we are back in the EB setup, and hence $\del_{\hat G}$ is approximately optimal. Thus $(\bbg_n,\del_{\hat G})$ is an asymptotic saddle point.

We can express this result from the point of view of the statistician:
\begin{theorem}
  If $\Del_{\hat G}=\bigl(\del_{\hat G}(Y_1),\dots,\del_{\hat G}(Y_n)\bigr)$, then for any $\bbg_n\in\scg$:
   \eqsplit{
    R(\bbg_n,\Del_{\hat G})  &\le \min_{\Del}\max_{\ti\bbg_n\in\scg} R(\ti\bbg_n,\Del)+\OP({\sqrt{\frac{\log n}{n}}}).
    }

\end{theorem}

Note that if Nature ``chooses'' some fixed $\th_{[n]}$, then we can simply take $\scg$ to be a point mass at $\bbg_n$.

We considered a specific restriction on Nature, namely $\bbg_n\in\scg$. We could replace it with other restrictions. For example, we could consider $\{\bbg_n\}\in\ti\scg$ if $\bbg_n\weakly G$, since then $R(\bbg_n,\Del)\to R(G,\Del)$ by the definition of weak convergence.

To summarize, it was known that the statistician can do as well as the Disabled Oracle. But since this oracle is disabled, he cannot define the bound---he is not an oracle in the true sense of the word (the Permutation Oracle is a true oracle, but he is better than can be mimicked).  We established a notion under which  procedures based on the Tweedie formula are minimax.

\appendix
\section{THe NPMLE is finitely supported}
\label{app:MLE}

The NPMLE is defined as 
 \eqsplit{
   \hat G= \argmax_G \summ i1n \log\int e^{-(y_i-\ti\th)^2/2\sig^2}dG(\th)
  }

\begin{proposition}
\label{prop:finiteSupport}
  $\hat G$ has at most $n$ support points.
\end{proposition} 
  
Consider any  parametric $\gamma_t:[-c,c]\to[-c,c]$ with $\gamma_0(\th)\equiv \th$. Then  the derivative of the log-likelihood of $\hat G_t=\hat G\circ \gamma_t$ with respect to $t$ should be 0 at 0. Consider $\gamma^\eps_t(\th')=\th'+t(\th+\eps-\th')\ind(\th<\th'<\th+\eps)$ for any $\th$ is the support of $\hat G$. Taking the limit as $\eps\to 0$,  we obtain that if  $\th$ is in the support of the NPMLE, then 
 \eqsplit{
    h(\th)&\equiv \summ i1n \frac{y_i-\th}{f_{\hat G}(y_i)} e^{y_i\th} = 0.
  }
Consider $h(\cdot)$ as a function of $\th$ (with $\hat G$ and $y_{[n]}$ fixed) it is of the form $\sum p_i(\th)e^{y_i\th}$, and thus it has at most $2n-1$ zeros by Lemma \ref{lem:polynomial} below, and since between any two maxima there is a minimum, the support of $\hat G$ has at most $n$ points which corresponds to the maximizers.
\begin{lemma}\label{lem:polynomial}
  The function $h(\th) = \summ i1n p_i(\th)e^{a_i\th}$, where $p_i$ is a polynomial of degree $r_i$ admits at most $\summ i1n (r_i+1)-1$ zeros. 
\end{lemma}
This result is quoted in our context in   \cite{Jewell82}, where it is referred  to \cite{PolaySzego1925}. I couldn't find it in the book (but the book has more than 400 pages), so, for completeness,  here is the very short argument using Descartes' and Rolle's rules in the way it would be proved in the book. See also \cite{Jameson2006}.  The proof is by induction on $n$. The case of $n=1$ is trivial. Assume it holds for $n=1,\dots,m$. Let $n=m+1$. By multiplying by $e^{-a_{m+1}\th}\ne 0$, it is clear that we can assume, wlog, that $a_{m+1}=0$. Thus,  $h(\th) = \summ i1n p_i(\th)e^{a_i\th}+ p_{m+1}(\th)$. By the induction assumption, its $r_{m+1}$ derivative has at most $\summ i1m (r_m+1)-1$ zeros. By Rolle's theorem (applied $r_{m+1}$ times) the proof is complete.
  
It can simplify the algorithm, in particular, the EM is considerably simplified, if we consider a fixed support on a $m$ support grid with spacing of $2c/m$. Let $\scg^m$ be the set of all pdf supported on this grid. Let $\hat\th_{[n]}$ and $\hat g_{[n]}$ be the NPMLE, $\ti\th_{[n]}$ be the sequence $\hat\th_{[n]}$ rounded to the grid, $\ti\th_j=m^{-1}\lfloor m\hat\th_j+0.5\rfloor$. Then the difference between the log-likelihood of the pair $(\hat \th_{[n]},\hat g_{[n]})$ and that of the pair $(\ti \th_{[n]},\ti g_{[n]})$ is only  
 \eqsplit{
    \ell(\hat \th_{[n]},\hat g_{[n]};Y_{[n]}) &-\ell(\ti \th_{[n]},\ti g_{[n]};Y_{[n]})
    \\ 
    &\le \ell(\hat \th_{[n]},\hat g_{[n]};Y_{[n]}) -\ell( \ti\th_{[n]},\hat g_{[n]};Y_{[n]})
    \\ 
    &\le n c^2\bigl(\sum_{j}|\hat \th_{j}-\ti \th_{j}|\bigr)^2 =\OP({ n^3/m^2}),
  }
since, the first derivative of the log-likelihood is 0 at the NPMLE, and the second derivative is the conditional covariance of $(\th_j,\th_k)$ given the observations (at the parameters).   Thus, if $m\gg n^{3/2}$ the difference is insignificant.

\bibliographystyle{imsart-nameyear} 
\bibliography{CDminimax.bib}

\end{document}